\documentclass[reqno]{alea2}
\usepackage{natbib}
\usepackage{fancyhdr}
\usepackage{graphicx}

\usepackage{amssymb}
\usepackage{amsmath}

\renewcommand{\cite}{\citet}

\makeatletter \@addtoreset{equation}{section} \makeatother

\renewcommand\theequation{\thesection.\arabic{equation}}
\renewcommand\thefigure{\thesection.\@arabic\c@figure}
\renewcommand\thetable{\thesection.\@arabic\c@table}

\theoremstyle{plain}
\newtheorem{theorem}{Theorem}[section]                                          
\newtheorem{proposition}[theorem]{Proposition}

\theoremstyle{definition}

\theoremstyle{remark}

\newcommand{\mc}[1]{{\mathcal #1}}

\newcommand{\bb}[1]{{\mathbb #1}}

\newcommand{\<}{\langle}
\renewcommand{\>}{\rangle}

\begin{document}


\author{Milton Jara}
\address{CEREMADE \\ Universit\'e Paris-Dauphine\\
Place du Mar\'echal de Lattre de Tassigny\\
Paris CEDEX 75775\\
France\\
}
\email{jara@ceremade.dauphine.fr}

\title[Hydrodynamics for inhomogeneous sytems]{Hydrodynamic limit of the exclusion process in inhomogeneous media} 

\begin{abstract}
We obtain the hydrodynamic limit of a simple exclusion process in an inhomogeneous environment of divergence form. Our main assumption is a suitable version of $\Gamma$-convergence for the environment. In this way we obtain an unified approach to recent works on the field.
\end{abstract}

\maketitle

\section{Introduction}

Since the seminal paper \cite{GPV}, the theory of hydrodynamic limit of interacting particle systems has evolved into a powerful tool in the study of non-equilibrium properties of statistical systems of many components (see the book \cite{KL} for a comprehensive exposition). Recently, and due to the infuence of physical and mathematical works about random walks in random environment, an increasing attention has been posed into particle systems evolving in random environments. Despite the early works \cite{Fri}, \cite{Qua}, \cite{Kou}, we mention \cite{FM}, \cite{Qua} \cite{Nag}, \cite{JL}, \cite{Fag}, \cite{GJ}, \cite{FJL}, \cite{FL}, \cite{Fag2}, \cite{GJ2}. In \cite{GJ}, \cite{JL} the {\em corrected empirical density} was introduced, which is nothing but a microscopic version of the compensated compactness lemma of Tartar \cite{Tar}. Roughly speaking, when the inhomogeneous environment (random or not) has a divergence form and has a $\Gamma$-limit, space homogenization of the environment and time homogenization of the interaction decouples, and the standard tools from the theory of hydrodynamic limit can be used to obtain the asymptotic behavior of the density of particles in a family of models, including the exclusion process and the zero-range process. 

In this review, we give an unified approach to this problem, recovering previous results in \cite{Nag}, \cite{JL}, \cite{Fag}, \cite{FJL}, \cite{Fag2}, in a simple way. In order to concentrate our efforts in the influence of the inhomogeneous environment on the asymptotics of the density of particles, we consider the simplest model of interacting particle systems, which is the symmetric exclusion process $\eta_t^n$ in an unoriented graph. In this process, particles perform symmetric random walks on a graph $\{X_n\}_n$ with some rates $\omega^n=\{\omega_{x,y}^n; x,y \in X_n\}$, conditioned to have at most one particle per site. We think of $\{X_n\}_n$ as a sequence of graphs embedding in some metric space $X$, and we are interested in the evolution of the measure $\pi_t^n(dx)$ in $X$, obtained by giving a mass $a_n^{-1}$ to each particle. 

This article is organized as follows. In Section \ref{s1} we give precise definitions of the exclusion process, the inhomogeneous environment and we state our main result. We also define what we mean by an approximation $\{X_n\}_n$ of $X$ and by $\Gamma$-convegence of the environment. In Section \ref{s2} we introduce the corrected empirical density and we prove our main theorem. In Section \ref{s3} we introduce the concept of energy solutions of the hydrodynamic equation, we prove uniqueness of such solutions and we obtain a substantial improvement of the main Theorem. The material of this Section is new and it gives a better understanding of the relation between $\Gamma$-convergence of the environment and hydrodynamic limit of the particle system. In Section \ref{s4} we discuss how to reobtain previous results in the literature relying in our main Theorem.

\section{Definitions and results}
\label{s1}
In this section we define the exclusion process in inhomogeneous environment and we recall some notions of $\Gamma$-convergence that will be necessary in order to obtain the hydrodynamic limit of this process.

\subsection{Partitions of the unity and approximating sequences}
\label{s1.1}
In this section we fix some notation and we define some objects which will be useful in the sequel. 
Let $(X,\mc B)$ be a Polish space. We assume that $X$ is $\sigma$-compact.
We say that a sequence of functions $\{\mc U_i; i \in I\}$ is a {\em partition of the unity} if:
\begin{description}
 \item[i)] for any $i \in I$, $\mc U_i:X \to [0,1]$ is a continuous function, 
 \item[ii)] for any $x \in X$, $\sum_{i \in I} \mc U_i(x)=1$,
 \item[iii)] for any $x \in X$, the set $\{i \in I; \mc U_i(x)>0\}$ is finite.
\end{description}

We say that the partition of the unity $\{\mc U_i;i \in I\}$ is {\em regular} if $\text{supp } \mc U_i$ is compact for any $i \in I$, and additionally $\mc U_i(X) = [0,1]$. 
We denote by $\mc M_+(X)$ the set of Radon, positive measures in $X$. The symbol $\{x_n\}_n$ will denote a sequence of elements $x_n$ in some space, indexed by the set $\bb N$ of positive integers.

Let $\{\mc U_i\}_i$ be a regular partition of the unity. We say that a sequence $\{x_i ; i \in I\}$ in $X$ is a {\em representative} of $\{\mc U_i\}_i$ if $\mc U_i(x_i)=1$ for any $i \in I$. Notice that we have $x_i \neq x_j$ for $i \neq j$. 

Let $\{\mc U_i^n; i \in I_n\}_n$ be a sequence of regular partitions of the unity. We say that a measure $\mu \in \mc M_+(X)$ is the scaling limit of the sequence $\{\mc U_i^n\}_n$ if there exists a sequence  $\{a_n\}_n$ of positive numbers such that for any sequence $\{x_i^n;i \in I_n\}$ of representatives of $\{\mc U_i^n\}_n$ we have
\[
\lim_{n \to \infty} \frac{1}{a_n} \sum_{i \in I_n} \delta_{x_i^n} = \mu
\]
with respect to the vague topology, where $\delta_x$ is the Dirac mass at $x \in X$. We call $\{a_n\}_n$ the {\em scaling} sequence. 

From now on, we fix a sequence $\{\mc U_i^n\}_n$ of regular partitions of the unity with scaling limit $\mu$, scaling sequence $\{a_n\}_n$ and we assume that $\mu(A)>0$ for any non-empty, open set $A \subseteq X$. 
Fix a sequence $\{x_i^n;i \in I_n\}$ of representatives of $\{\mc U_i^n\}_n$. Define $X_n = \{x_i^n; i \in I_n\}$. 
Since $\{\mc U_i^n\}$ is a partition of the unity, the induced topology in $X_n$ coincides with the discrete topology. For $x=x_i^n$, we will denote $\mc U_x^n= \mc U_i^n$. Define
\[
\mu_n(dx) = \frac{1}{a_n} \sum_{x \in X_n} \delta_x(dx).
\]

By definition, $\mu_n \to \mu$ in the vague topology. We denote by $\mc L^2(\mu_n)$ the Hilbert space of functions $f: X_n \to \bb R$ such that $\sum_{x \in X_n} f(x)^2 <+\infty$, equipped with the inner product
\[
\< f, g\>_n = \frac{1}{a_n} \sum_{x \in X_n} f(x)g(x).
\]

We define $\mc L^2(\mu)$, $\mc L^1(X_n)$ and $\mc L^1(\mu)$ in the analogous way and we denote $\<f,g\> = \int fg d\mu$.
We denote by $\mc C_c(X)$ the set of continuous functions $f: X \to \bb R$ with compact support. In the same spirit, we denote by $\mc C_c(X_n)$ the set of functions $f: X_n \to \bb R$ with finite support. We define the projection $S_n: \mc C_c(X) \to \mc C_c(X_n)$ by taking
\[
\big( S_n G\big)(x) = a_n \int G \mc U_x^n d\mu.
\]

This operator, under suitable conditions, can be extended to a bounded operator from $\mc L^2(X)$ to $\mc L^2(X_n)$. Notice that  $\int S_n G d\mu_n = \int G d\mu$. Therefore $S_n$ is continuous from $\mc L^1(\mu)$ to $\mc L^1(X_n)$.

\subsection{$\Gamma$-convergence}

Define $\bar{ \bb R} = [-\infty,+\infty]$. Let $(Y,\mc F)$ be a topological space, and let $F_n,F: Y \to \bar{ \bb R}$. We say that $F_n$ is $\Gamma$-convergent to $F$ if:
\begin{description}
  \item[i)] For any sequence $\{y_n\}_n$ in $Y$ converging to $y \in Y$,
  \[
  F(y) \leq \liminf_{n \to \infty} F_n(y_n).
  \]
  \item[ii)] For any $y \in Y$ there exists a sequence $\{y_n\}_n$ converging to $y$ such that
  \[
  \limsup_{n \to \infty} F_n(y_n) \leq F(y).
  \]
\end{description}

An important property of $\Gamma$-convergence is that it implies {\em convergence of minimizers} in the following sense:

\begin{proposition}
\label{p1}
Let $F_n, F: Y \to \bar{\bb R}$ be such that $F_n$ is $\Gamma$-convergent to $F$. Assume that there exists a relatively compact set $K \subseteq Y$ such that for any $n$,
\[
\inf_{y \in Y} F_n(y) =\inf_{y \in K} F_n(y).
\]

Then,
\[
\lim_{n \to \infty} \inf_{y \in K} F_n(y) = \min_{y \in Y} F(y).
\]
Moreover, if $\{y_n\}_n$ is a sequence in $K$ such that $\lim_n (F_n(y_n)-\inf_K F_n)=0$, then any limit point $y$ of $\{y_n\}_n$ satisfies $F(y) = \min_Y F$.
\end{proposition}

A useful property that follows easily from the definition is the stability of $\Gamma$-convergence under continuous perturbations:

\begin{proposition}
\label{p2}
Let $F_n,F: Y \to \bar{\bb R}$ be such that $F_n$ is $\Gamma$-convergent to $F$. Let $G_n: Y \to \bb R$ be such that $G_n$ converges uniformly to a continuous limit $G$. Then, $F_n+G_n$ is $\Gamma$-convergent to $F+G$.
\end{proposition}

\subsection{The exclusion process in inhomogeneous environment}

In this section we define the exclusion process in inhomogeneous environment as a system of particles evolving in the set $X_n$. 
Let $\omega^n = \{\omega_{x,y}^n;x,y \in X_n\}$ be a sequence of non-negative numbers such that $\omega^n_{x,x}=0$ and $\omega^n_{x,y}=\omega^n_{y,x}$ for any $x,y \in X_n$. We call $\omega^n$ the {\em environment}. We define the exclusion process $\eta_t^n$ with environment $\omega^n$ as a continuous-time Markov chain of state space $\Omega_n = \{0,1\}^{X_n}$ and generated by the operator
\[
L_n f(\eta) = \sum_{x,y \in X_n} \omega_{x,y}^n \big[f(\eta^{x,y}) - f(\eta)\big],
\]
where $\eta$ is a generic element of $\Omega_n$, $f: \Omega_n \to \bb R$ is a function which depends on $\eta(x)$ for a finite number of elements $x \in X_n$ (that is, $f$ is a {\em local function}) and $\eta^{x,y} \in \Omega_n$ is defined by
\[
\eta^{x,y}(z) = 
\begin{cases}
\eta(y), & \text{if } z=x\\
\eta(x), & \text{if } z=y\\
\eta(z), & \text{if } z \neq x,y.\\
\end{cases}
\]

In order to have a well-defined Markovian evolution for any initial distribution $\eta_0^n$, we assume that $\sup_x \sum_{y \in X_n} \omega^n_{x,y} <+\infty$.
We interpret $X_n$ as a set of sites and $\eta_t^n(x)$ as the number of particles at site $x \in X_n$ at time $t$. Since $\eta_t^n(x) \in \{0,1\}$, there is at most one particle per site at any given time: this is the so-called {\em exclusion rule}. Notice that the dynamics is conservative in the sense that no particles are annihilated or destroyed.

Our interest is to study the collective behavior of particles for the sequence of processes $\{\eta_\cdot^n\}_n$. In order to do this, we introduce the {\em empirical density of particles} as the measure-valued process $\pi_t^n$ defined by
\[
\pi_t^n(G) = \frac{1}{a_n} \sum_{x \in X_n} \eta_t^n(x) S_n G(x)
\]
for any $G \in \mc C_c(X)$. Using Riesz's theorem, it is not difficult to check that $\pi_t^n$ is effectively a positive Radon measure in $X$.  
Observe that when $\eta_0^n(x)=1$ for any $x \in X_n$, then $\eta_t^n(x)=1$ for any $x \in X_n$ and any $t \geq 0$. In this situation, the empirical process $\pi_t^n$ is identically equal to the measure $\mu$.  Notice that the random  variable $\pi_t^n$ defined in this way corresponds to a  process defined in the space $\mc D([0,\infty), \mc M_+(X))$ of c\`adl\`ag paths with values in $\mc M_+(X)$. For functions $G:X_n \to \bb R$, we define $\pi_t^n(G) = a_n^{-1} \sum_x \eta_t^n(x) G(x)$.

\subsection{$\Gamma$-convergence of the environment}

In this section we will make a set of assumptions on the environment $\{\omega^n\}_n$ which will allows us to obtain an asymptotic result for the sequence $\{\pi_\cdot^n\}_n$. We start with two assumptions about the sequence of partitions of the unity $\{\mc U_x^n\}_n$. Our first assumption corresponds to a sort of ellipticity condition on the partitions of the unity $\{\mc U_x^n\}_n$: 

\begin{description}
\item[\bf (H1)] There exists $\Theta <+\infty$ such that 
\[
\sup_{x \in X_n} a_n \int \mc U_x^n d\mu \leq\Theta \text{ for any $n>0$.}
\]

\end{description}

Under this condition, the projection $S_n$ satisfies $||S_n G||_\infty \leq \theta ||G||_\infty$, and by interpolation $S_n$ can be extended to a continuous operator from $\mc L^2(\mu)$ to $\mc L^2(X_n)$. Our second condition states that $S_n$ is close to an isometry when $n \to \infty$:
\begin{description}
  \item[\bf{(H2)}] For any $F \in \mc L^2(\mu)$, we have
  \[
  \lim_{n \to \infty} \<S_n F, S_n F\>_n = \<F,F\>.
  \] 
\end{description}

Now we are ready to discuss on which sense we will say that the environment $\omega^n$ converges. 
For a given function $F: X_n \to \bb R$ of finite support, we define $\mc L_n F$ by
\[
\mc L_n F(x) = \sum_{y \in X_n} \omega_{x,y}^n \big(F(y)-F(x)\big).
\]

It turns out that $\mc L_n$ can be extended to a non-positive operator in $\mc L^2(X_n)$. In fact, for any function $F$ of finite support, the {\em Dirichlet form}
\[
\<F, - \mc L_n F\>_n = \frac{1}{2 a_n} \sum_{x, y \in X_n} \omega_{x,y}^n \big(F(y)-F(x)\big)^2
\]
is clearly non-negative. For a function $G \in \mc L^2(\mu)$, define $\mc E_n(G) = \<S_n G, -\mc L_n S_n G\>$. Notice that $\mc E_n: \mc L^2(\mu) \to \bar{\bb R}$ is a quadratic form. Now we are ready to state our first hypothesis about the environment:
\begin{description}
  \item[{\bf (H3)}] There exists a non-negative, symmetric operator $\mc L: D(\mc L) \subseteq \mc L^2(\mu) \to \mc L^2(\mu)$ such that $\mc E_n$ is $\Gamma$-convergent to $\mc E$, where $\mc E(G) = -\int G \mc L G d\mu$. 
\end{description}

Our second hypothesis about the environment $\omega^n$ concerns to its $\Gamma$-limit $\mc L$:
\begin{description}
  \item[{\bf (H4)}] There exists a dense set $\mc K \subseteq \mc C_c(X)$ such that $\mc K$ is a kernel for the operator $\mc L$, and for any $G \in \mc K$, $\mc LG$ is continuous and $\int |\mc L G| d\mu <+\infty$.
\end{description}

\subsection{Hydrodynamic limit of $\eta_t^n$}

In this section we explain what we understand as the hydrodynamic limit of $\eta_t^n$. We say that a sequence $\{\nu_n\}_n$ of distributions in $\Omega_n$ is {\em associated } to a function $u:X \to \bb R$ if for any function $G \in \mc C_c(X)$ and any $\epsilon >0$ we have
\[
\lim_{n \to \infty} \nu_n \Big\{ \Big|\frac{1}{a_n} \sum_{x \in X_n} \eta(x) G(x) - \int G(x)u(x)\mu(dx)\Big|>\epsilon\Big\}=0. 
\]

Notice that we necessarily have $0 \leq u(x) \leq 1$ for any $x \in X$, since $\eta(x) \in \{0,1\}$.
Fix an initial profile $u_0:X \to [0,1]$ and take a sequence of distributions $\{\nu_n\}$ associated to $u_0$. Let $\eta_t^n$ be the exclusion process with initial distribution $\nu_n$. We denote by $\bb P_n$ the law of $\eta_t^n$ in $\mc D([0,\infty),\Omega_n)$ and by $\bb E_n$ the expectation with respect to $\bb P_n$. The fact that $\{\nu_n\}_n$ is associated to $u_0$ can be interpreted as a law of large numbers for the empirical measure $\pi_0^n$: $\pi_0^n(dx)$ converges in probability to the deterministic measure $u_0(x) \mu(dx)$. We say that the hydrodynamic limit of $\eta_t^n$ is given by the equation $\partial_t u = \mc L u$ if for any $t>0$, the empirical measure $\pi_t^n(dx)$ converges in probability to the measure $u(t,x) \mu(dx)$, where $u(t,x)$ is the solution of the equation $\partial_t u = \mc L u$ with initial condition $u_0$. Before stating our main result in a more precise way, we need some definitions. 

For $F, G \in D(\mc L)$, define the bilinear form $\mc E(F,G) = - \int F \mc L G d\mu$. Notice that $\mc E(F,G)$ is still well defined if only $G \in D(\mc L)$. 
We say that a function $u:[0,T] \times X \to [0,1]$ is a weak solution of (\ref{echid}) with initial condition $u_0$ if $\int_0^T \int u_t^2 d\mu dt <+\infty$ and for any differentiable path $G: [0,T] \to \mc K$ such that $G_T \equiv 0$ we have
\[
\<u_0,G_0\> +\int_0^T \Big\{ \<\partial_t G_t,u_t\>  -\mc E(G_t,u_t)\Big\}dt=0.
\]

\begin{theorem}
\label{t1}
Let $\{\nu_n\}_n$ be associated to $u_0$ and consider the exclusion process $\eta_t^n$ with initial distribution $\nu_n$. Assume that $\int \pi_0^n(dx)$ is uniformly finite:
\begin{description}
  \item[{\bf (H5)}] 
  \[
  \lim_{M \to \infty} \sup_n \nu_n\Big\{\frac{1}{a_n} \sum_{x \in X_n} \eta(x) >M \Big\} =0.
  \]
\end{description}

Then, the sequence of processes $\{\pi_\cdot^n(dx)\}_n$ is tight and the limit points are concentrated on measures of the form $u(t,x)\mu(dx)$, where $u(t,x)$ is a weak solution of the {\em hydrodynamic equation}
\begin{equation}
\label{echid}
\left\{
\begin{array}{rcl}
\partial_t u  & = & \mc L u, \\
u(0,\cdot) & = & u_0(\cdot).\\
\end{array}
\right.
\end{equation}

If such solution is unique, the process $\pi_\cdot^n(dx)$ converges in probability with respect to the Skorohod topology of $\mc D([0,\infty),\mc M_+(X))$ to the deterministic trajectory $u(t,x)\mu(dx)$.
\end{theorem}

Usually in the literature, hydrodynamic limits are obtained in finite volume, since the pass from finite to infinite volume is non-trivial. Assumption {\bf (H5)} is in this spirit: it is automatically satisfied when the cardinality of $X_n$ is of the order of $a_n$ (on which case $\mu(X)<+\infty$), and it is very restrictive when $X_n$ is infinite. For simplicity, we restrict ourselves to the case on which {\bf (H5)} is satisfied.

\section{Hydrodynamic limit of $\eta_t^n$: proofs}
\label{s2}
In this section we obtain the hydrodynamic limit of the process $\eta_t^n$. The strategy of proof of this result is the usual one for convergence of stochastic processes. First we prove tightness of the sequence of processes $\{\pi_\cdot^n\}_n$. Then we prove that any limit point of this sequence is concentrated on solutions of the hydrodynamic equation. Finally, a uniqueness result for such solutions allows us to conclude the proof. However, the strategy outlined above will not be carried out for $\{\pi_\cdot^n\}_n$ directly, but for another process $\hat \pi_\cdot^n$, which we call the {\em corrected} empirical process.

\subsection{The corrected empirical measure}
\label{s2.1}

In this section we define the so-called corrected empirical measure, relying on the $\Gamma$-convergence of the environment. First we need to extract some information about convergence of the operators $\mc L_n$ to $\mc L$ from the $\Gamma$-convergence of the associated Dirichlet forms.

Take a general Hilbert space $\mc H$ and let $\mc A$ be a non-negative, symmetric operator defined in $\mc H$. By Lax-Milgram theorem, we know that for any $\lambda>0$ and any $g \in \mc H$, the equation $(\lambda+\mc A)f =g$ has a unique solution in $\mc H$. Moreover, the solution $f$ is the minimizer of the functional $f \mapsto \<f,\mc A f\>+\lambda||f||^2-2\<f,g\>$. Fix $\lambda>0$. For a given function $G \in \mc L^2(\mu)$, define the functionals
\[
\mc E_n^G(F) = \mc E_n(F) +\lambda\<S_nF,S_nF\>_n -2 \<S_n F, S_n G\>_n,
\]
\[
\mc E^G(F) = \mc E(F) +\lambda\<F,F\>-2 \<F, G\>.
\]

By Proposition \ref{p2}, $\mc E_n^G$ is $\Gamma$-convergent to $\mc E^G$. In particular, a sequence of minimizers $F_n$ of $\mc E_n^G$ converge to the minimizer $F$ of $\mc E^G$. Notice that $F_n$ is not uniquely defined in general, although $S_n F_n$ it is. By the discussion above, $(\lambda-\mc L_n) S_n F_n = S_n G$ and $(\lambda -\mc L) F =G$. Since the operator norm of $S_n$ is bounded by $\Theta$, we conclude that the $\mc L^2(X_n)$-norm of $S_n F_n-S_n F$ converges to 0 as $n \to \infty$.
By {\bf (H2)}, we conclude that $\mc E_n(F_n)$ converges to $\mc E(F)$.

Now we are ready to define the corrected empirical measure $\hat \pi_t^n$. Take a function $G \in \mc K$ and define $H=(\lambda-\mc L) G$. Define $G_n$ as a minimizer of $\mc E_n^H$. Notice that in this way $S_n G_n$ is uniquely defined. Then we define 
\[
\hat \pi_t^n(G) = \frac{1}{a_n} \sum_{x \in X_n} \eta_t^n(x) S_n G_n(x).
\]

In order to prove that $\hat \pi_t^n(G)$ is well defined, we need to prove that $\sum_x S_n G_n(x)$ is finite. Remember that $(\lambda - \mc L_n) S_n G_n = S_n H$. Consider the continuous-time random walk with jump rates $\omega_{x,y}^n$. Remember that the condition $\sup_x \sum_y \omega_{x,y}^n$ ensures that this random walk is well defined. Let $p_t^n(x,y)$ be its transition probability function. An explicit formula for $S_n G_n$ in terms of $p_t^n(x,y)$ is
\[
S_n G_n (x) = \int_0^\infty e^{-\lambda t} \sum_{y \in X_n} p_t^n(x,y) S_n H(y) dt.
\]

Since $\sum_x p_t(x,y)=1$ for any $y \in X_t$, we conclude that
\[
\frac{1}{a_n} \sum_{x \in X_n} S_n G_n(x) = \frac{1}{\lambda} \int H d\mu
\]
and in particular $S_n G_n$ is summable. We conclude that $\hat \pi_t^n(G)$ is well defined. Notice that it is not clear at all if $\hat \pi_t^n$ is well defined as a measure in $X$.  

\subsection{Tightness of $\{\pi_\cdot^n\}_n$ and proof of Theorem \ref{t1}}

In this section we prove tightness of $\{\pi_\cdot^n\}_n$ and we prove Theorem \ref{t1}. As we will see, we rely on the corrected empirical measure, which turns out to be the right object to be studied. 
By {\bf (H5)}, we have
\[
\lim_{n \to \infty} \bb P_n \Big( \sup_{0 \leq t <+\infty} \big|\pi_t^n(G) - \hat \pi_t^n(G)\big|>\epsilon\Big) =0.
\]

Notice that {\bf (H5)} can be substituted by the following condition, which can be sometimes proved directly.
\begin{description}
  \item[{\bf (H5')}] For any $G \in \mc K$,
  \[
  \lim_{n \to \infty} \frac{1}{a_n} \sum_{x \in X_n} \big| S_n G_n(x) - S_n G(x)\big| =0.
  \]
\end{description}

In particular, $\{\pi_\cdot^n(G)\}_n$ is tight if and only if $\{\hat \pi_\cdot^n(G)\}_n$ is tight. The usual way of proving tightness of $\{\hat \pi_\cdot^n(G)\}_n$ is to use a proper martingale decomposition. A simple computation based on Dynkin's formula shows that
\begin{equation}
\label{ec1}
\mc M_t^n(G) = \hat \pi_t^n(G) - \hat \pi_0^n(G) - \int_0^t \pi_s^n(\mc L_n S_n G_n) ds
\end{equation}
is a martingale. The quadratic variation of$\mc M_t^n(G)$ is given by
\[
\<\mc M_t^n(G)\> = \int_0^t \frac{1}{a_n^2} \sum_{x,y \in X_n} \big(\eta_s^n(y)-\eta_s^n(x)\big)^2 \omega^n_{x,y} \big( S_n G_n(y) -S_n G_n(x)\big)^2 ds.
\]
In particular, $\<\mc M_t^n(G)\> \leq t a_n^{-1} \mc E_n(G_n)$. At this point, the convenience of introducing the corrected empirical process becomes evident. By definition, $\mc L_n S_n G_n = S_n \mc L G +\lambda(S_n G_n - S_n G)$. Since $H= (\lambda-\mc L)G$, the function $G$ is the minimizer of $\mc E^H$. Therefore, $G_n$ converges to $G$ in $\mc L^2(X)$. By {\bf (H2)}, the $\mc L^2(X_n)$-norm of $S_n G_n -S_n G$ goes to 0 and $\mc E_n(G_n)$ converges to $\mc E(G)$. 

We conclude that $\mc M_t^n(G)$ converges to 0 as $n \to \infty$, and in particular the sequence $\{\mc M_\cdot^n(G)\}_n$ is tight. In the other hand, the integral term in (\ref{ec1}) is equal to
$\int_0^t \pi_s^n(\mc L G)ds$.

Notice that $\pi_s^n(\mc L G) \leq \int |\mc L G| d \mu$ for any $t \geq 0$, from where we conclude that the integral term is of bounded variation, uniformly in $n$. Tightness follows at once. Since $\{\hat \pi_0^n(G)\}_n$ is tight by comparison with $\{\pi_0^n(G)\}_n$, we conclude that $\{\hat \pi_\cdot^n(G)\}_n$ is tight, which proves the first part of Theorem \ref{t1}. As a by-product, we have obtained tightness for $\{\pi_\cdot^n\}_n$ as well, and the convergence result
\[
\lim_{n \to \infty}\Big\{ \pi_t^n(G) -\pi_0^n(G) - \int_0^t \pi_s^n(\mc L G) ds\Big\} =0
\]
for any $G \in \mc K$. Notice that we have exchanged $\hat \pi_t^n(G)$ by $\pi_t^n(G)$. Let $\pi_\cdot$ be a limit point of $\{\pi_\cdot^n\}_n$. Then, $\pi_\cdot$ satisfies the identity
\[
\pi_t(G) -\pi_0(G) - \int_0^t \pi_s(\mc L G) ds =0
\] 
for any function $G \in \mc K$. By hypothesis, $\pi_0(dx) = u_0(x)\mu(dx)$. Repeating the arguments for a function $G_t(x) = G_0(x) + t G_1(x)$ with $G_0,G_1 \in \mc K$, we can prove that
\[
\pi_t(G_t) - \pi_0(G_0) -\int_0^t \pi_s((\partial_t+\mc L)G_s)ds=0
\] 
for any piecewise-linear trajectory $G_\cdot:[0,T] \to \mc K$. The same identity holds by approximation for any smooth path $G_\cdot: [0,T] \to \mc C_c(X)$, which proves that the process $\pi_\cdot$ is concentrated on weak solutions of the hydrodynamic equation. When such solutions are unique, the process $\pi$ is just a $\delta$-distribution concentrated on the path $u(t,x)\mu(dx)$. Since compactness plus uniqueness of limit points imply convergence, Theorem \ref{t1} is proved. 

\section{Energy solutions and energy estimate}
\label{s3}
In this section we define what we mean by {\em energy solutions} of Equation (\ref{echid}), we prove that any limit point of the empirical measure $\{\pi_\cdot^n\}$ is concentrated on energy solutions of (\ref{echid}) and we give a simple criterion for uniqueness of such solutions. 

\subsection{Energy solutions} 

Let $\mc E : H \to \bar{\bb R}$ be a quadratic form defined over a Hilbert space $H$ of inner product $\<\cdot,\cdot\>$. We say that $\mc E$ is {\em closable} if for any sequence $\{f_n\}_n$  converging in $H$ to some limit $f$ such that $\mc E(f_n-f_m)$ goes to $0$ as $n, m \to \infty$, we have $f=0$. 
Let $\mc E: H \to \bar{\bb R}$ be closable. We define $\mc H_1=\mc H_1(\mc E)$ as the closure of the set $\{f\in H; \mc E(f)<+\infty\}$ under the norm $||f||_1 = (\mc E(f)+ \<f,f\>)^{1/2}$. 

We say that a dense set $K \subseteq H$ is a {\em kernel} of $\mc E$ if $\mc H_1$ is equal to the closure of $K$ under the norm $||\cdot||_1$.  We say that a  symmetric operator $\mc L: D(\mc L) \subseteq H \to H$ generates $\mc E$ if $\mc E(f)=\<f,-\mc Lf\>$ for $f \in D(\mc L)$ and $D(\mc L)$ is a kernel of $\mc E$. 

Fix $T >0$. For a function $u: [0,T] \to H$ we define the norm
\[
||u||_{1,T} = \Big( \int_0^T ||u_t||_1^2 dt \Big)^{1/2}
\]
and we define $\mc H_{1,T}$ as the Hilbert space generated by this norm. 
Given a closable form $\mc E$ generated by the operator $\mc L$, we say that a trajectory $u: [0,T] \to H$ is an {\em energy solution} of (\ref{echid}) if $u \in \mc H_{1,T}$ and for any differentiable trajectory $G: [0,T] \to \mc H_1$  with $G(T)=0$ we have
\[
\<G_0, u_0\> + \int_0^T \Big\{ \<\partial_t G_t, u_t\> - \mc E(G_t,u_t)\Big\} dt =0.
\]

In other words, an energy solution of (\ref{echid}) is basically a weak solution belonging to $\mc H_{1,T}$. In fact, by taking suitable approximations of $G$, it is enough to prove this identity for trajectories $G$ such that $G_t \in K$ for any $t \in [0,T]$, where $K$ is any kernel of $\mc E$ contained in $D(\mc L)$. Notice that the norm in $\mc H_{1,T}$ is stronger than the norm $\int_0^T u_t^2 dt$, and therefore a weak solution is effectively weaker than an energy solution of (\ref{echid}).

\subsection{The energy estimate}

In this section we prove that the limit points of the empirical measure are concentrated on energy solutions of (\ref{echid}). For simplicity, we work on finite volume. From now on we assume that $X$ is compact. Therefore, there exists a constant $\kappa$ such that the cadinality of $X_n$ is bounded by $\kappa a_n$. We have the following estimate.

\begin{theorem}
\label{t2}
Fix $T >0$. Let $\{H^i: X_n \times X_n \times [0,T] \to \bb R; i=1,\dots,l\}$ be a finite sequence of functions. There exists a constant $C=C(T)$ such that
\begin{multline}
\label{en.est}
\bb E_n \Big[ \sup_{i=1,\dots,l} \int_0^T  \Big\{ \frac{2}{a_n} \sum_{x,y \in X_n} \omega^n_{x,y} H_{x,y}^i(t) \big( \eta_t^n(y) -\eta_t^n(x)\big) \\
	-\frac{1}{a_n} \sum_{x,y \in X_n} \omega_{x,y}^n (H_{x,y}^i)^2 \eta_t^n(x)\Big\} dt \Big]
\leq C + \frac{\log l}{a_n}.
\end{multline}
\end{theorem}

\begin{proof}

Before starting the proof of this theorem, we need some definitions. Fix $\rho >0$. Denote by $\nu^\rho$ the product measure in $\Omega_n$ defined by 
\[
\nu^\rho\big(\eta(x_1)=1,\dots,\eta(x_k)=1\big) = \rho^k.
\]

It is not difficult to check that the measure $\nu^\rho$ is left invariant under the evolution of $\eta_t$. For two given probability measures $P_1$, $P_2$, we define the entropy $H(P_1|P_2)$ of $P_1$ with respect to  $P_2$ as
\[
H(P_1|P_2)= 
\begin{cases}
+ \infty, & \text{ if $P_1$ is not absolutely continuous with respect to $P_2$}\\
\int \log \frac{dP_1}{dP_2} d P_1 & \text{ otherwise. }\\
\end{cases}
\]

For $\eta \in \Omega_n$, denote by $\delta_\eta$ the Dirac measure at $\eta$. It is not difficult to see that $H(\delta_\eta|\nu^\rho) \leq C(\rho) a_n$ for any $\eta \in \Omega_n$, where $C(\rho)$ is a constant that can be chosen independently from $n$. Let us denote by $\bb P^\rho$ the distribution in $D([0,T],\Omega_n)$ of the process $\eta_t^n$ with initial distribution $\nu^\rho$. By the convexity of the entropy, $H(\bb P_n|\bb P^\rho) \leq C(\rho,T)a_n$ for a constant $C(\rho,T)$ not depending on $n$. The following arguments are standard and can be found in full rigor in \cite{KL}. Let us denote by $F^i(s)$ the function (depending on $H^i(s)$ and $\eta_s^n$) under the time integral in \eqref{en.est}. By the entropy estimate,
\begin{align*}
\bb E_n \Big[ \sup_{i=1,\dots,l} \int_0^T  F^i(t)dt\Big]
	\leq \frac{H(\bb P_n|\bb P^\rho)}{a_n} +
	\frac{1}{a_n} \log \bb E^\rho \Big[ \exp\big\{ \sup_{i=1,\dots,l} a_n \int_0^T F^i(t)dt\big\}\Big].
\end{align*}

In order to take the supremum out of the expectation, we use the inequalities $\exp\{\sup_i b_i\} \leq \sum_i\exp\{b_i\}$ and $\log\{\sum_i b_i\} \leq \log l + \sup_i \log b_i$, valid for any real numbers $\{b_i, i=1,\dots,l\}$. In this way we obtain the bound
\begin{equation}
 \label{ec2}
\begin{split}
\bb E_n \Big[ \sup_{i=1,\dots,l} \int_0^T  F^i(t)dt\Big]
	&\leq C(\rho,T) +\frac{\log l}{a_n} \\
	&\quad+\sup_{i=1,\dots,l} \frac{1}{a_n} \log \bb E^\rho \Big[ \exp\big\{  a_n\int_0^T F^i(t)dt\big\}\Big].
\end{split}
\end{equation}

Therefore, it is left to prove that the last supremum is not positive. It is enough to prove that the expectation $ \bb E^\rho \big[ \exp\big\{  \int_0^T F^i(t)dt\big\}\big]$ is less or equal than $1$ for any function $F^i$. From now on we drop the index $i$. By Feynman-Kac's formula plus the variational formula for the largest eigenvalue of the operator $F(t)+L_n$, we have
\[
\frac{1}{a_n} \log \bb E^\rho \Big[ \exp\big\{  a_n\int_0^T F(t)dt\big\}\Big]
	\leq \int_0^T \sup_f \big\{ \<F(t),f^2\>_\rho -\<f, -L_n f\>_\rho\},
\]
where we have denoted by $\<\cdot,\cdot\>_\rho$ the inner product in $\mc L^2(\nu_\rho)$ and the supremum is over functions $f \in \mc L^2(\nu_\rho)$. A simple computation using the invariance of $\nu_\rho$ shows that
\[
\<f, -L_n f\>_\rho = \sum_{x,y \in X_n} \omega_{x,y}^n \int \big[f(\eta^{x,y})-f(\eta)\big]^2 \nu_\rho(d\eta).
\]
Recall the expression for $F(t)$ in terms of $H$. We will estimate each term of the form $2 a_n^{-1} \<H_{x,y}(\eta(y)-\eta(x)),f^2\>_\rho$ separatedly:
\begin{align*}
\frac{2}{a_n} \<H_{x,y}(\eta(y)-\eta(x)),f^2\>_\rho
	&= \frac{2}{a_n} H_{x,y} \<\eta(x), f(\eta^{x,y})^2-f(\eta)^2\>_\rho  \\
	& \leq \frac{2}{a_n} \Big\{ \frac{(H_{x,y})^2\beta_{x,y}^n}{2} \<\eta(x), (f(\eta^{x,y})+f(\eta))^2\>_\rho \\
	&\quad+ \frac{1}{2 \beta_{x,y}^n} \<\eta(x),(f(\eta^{x,y})-f(\eta))^2\>_\rho \Big\}.
\end{align*}

Choosing $\beta_{x,y}^n = 1/\omega_{x,y}^n$ and putting this estimate back into (\ref{ec2}), we obtain the desired estimate.
\qed
\end{proof}

Take $G^i \in \mc K$ and take $H_{x,y}^i = S_n G_n^i(y) - S_n G_n^i(x)$, with $G_n^i$ defined as in Section \ref{s2.1}. Recall the identity $\mc L_n S_n G_n^i= S_n \mc L G^i + \lambda(S_n G_n^i -S_n G^i)$. The energy estimate (\ref{en.est}) gives
\[
 \bb E_n \Big[ \sup_{i=1,\dots,l} \int_0^T \big(2 \hat \pi_t^n(\mc L G^i) - \mc E_n(G_n^i) \big)dt \Big] \leq C(\rho,T) + C_1(l,n),
\]
where $C_1(l,n)$ is a constant that goes to 0 when $l$ is fixed and $n \to \infty$. Take a limit point of the sequence $\{\pi_\cdot^n\}_n$. We have already seen that $\hat \pi_t^n(\mc L G^i)$ converges to $\pi_t(\mc L G)$. Therefore, the process $\pi_\cdot$ satisfies
\[
 E \Big[ \sup_{i=1,\dots,l} \int_0^T \big(2 \pi_s(\mc L G^i) - \mc E(G^i)\big)dt\Big] \leq C(\rho,T).
\]

Similar arguments prove that for piecewise linear trajectories $\{G^i_t; i=1,\dots,l\}$ in $\mc K$, we have
\[
 E \Big[ \sup_{i=1,\dots,l} \int_0^T \big(2 \pi_s(\mc L G^i(t)) - \mc E(G^i(t))\big)dt\Big] \leq C(\rho,T).
\]

Since $l$ is arbitrary and piecewise linear trajectories with values in $\mc K$ are dense in $\mc H_{1,T}$, we conclude that $E[||\pi_\cdot||_{1,T}^2] <+\infty$, from where we conclude that $||\pi_\cdot||_{1,T}$ is finite $a.s.$
We establish this result as a theorem.

\begin{theorem}
 \label{t3}
 Let $\eta_t^n$ an exclusion process as in Theorem \ref{t1}.
 If one of the following conditions is satisfied,
 \begin{enumerate}
 \item[i)]
 $X$ is compact,
 \item[ii)]
 Assumption {\bf (H5')} holds and the entropy density is finite:
 \[
  \sup_{n} \frac{H(\bb P_n| \bb P^\rho)}{a_n} <+\infty,
 \]
\end{enumerate}
then any limit point of the sequence $\{\pi^n_\cdot(dx)\}_n$ is concentrated on energy solutions of the hydrodynamic equation (\ref{echid}). In particular, since such energy solutions are unique, the sequence $\{\pi_\cdot^n(dx)\}_n$ is convergent.
\end{theorem}

\subsection{Uniqueness of energy solutions}

In this section we prove uniqueness of energy solutions for (\ref{echid}). Since the equation is linear, it is enough to prove uniqueness for the case $u_0 \equiv 0$. Let $u_t$ be a solution of (\ref{echid}) with $u_0 \equiv 0$. Then,
\[
 \int_0^T \big\{ \<\partial_t G_t , u_t\> - \mc E(G_t,u_t)\big\}dt =0
\]
for any differentiable trajectory in $\mc H_{1,T}$ with $G_T=0$. Take $G_t = - \int_t^T u_s ds$. Then $\partial_t G_t = u_t$ and the first term above is equal to $\int_0^T\<u_t,u_t\>dt$. An approximation procedure and Fubini's theorem shows that the second term above is equal to
\[
 \frac{1}{2} \mc E\Big(\int_0^T u_t dt\Big).
\]

Both terms are non-negative, so we conclude that $\int_0^T \<u_t,u_t\>dt =0$ and $u_t \equiv 0$.

\section{Applications}
\label{s4}
In this section we give some examples of systems on which Theorems \ref{t1} and \ref{t3} apply. In the literature, the sequence $\omega^n$ is often referred as the set of {\em conductances} of the model. Unless stated explicitely, in these examples, $X$ will be equal to $\bb R^d$ or the torus $\bb T^d = \bb R^d / \bb Z^d$. The set $X_n$ will be equal to $n^{-1} \bb Z^d$ and we construct the partitions $\{\mc U_x^n\}$ in the canonical way, taking $\mc  U_x^n$ as a continuous, piecewise linear function with $\mc U_x^n(x)=1$ and $\mc U_x^n(y)=0$ for $y\in X_n$, $y \neq x$.

\subsection{Homogenization of ergodic, elliptic environments}

Let $(\Omega,\mc F,P)$ be a probability space. Let $\{\tau_x; x \in \bb Z^d\}$ be a family of $\mc F$-mesurable maps $\tau_x:\Omega \to \Omega$ such that
\begin{description}
\item[i)] $P(\tau_x^{-1} A) = P(A)$ for any $A \in \mc F$, $x \in \bb Z^d$.
\item[ii)] $\tau_x \tau_{x'} = \tau_{x+x'}$ for any $x,x' \in \bb Z^d$.
\item[iii)] If $\tau_x A =A$ for any $x \in \bb Z^d$, then $P(A)=0$ or $1$.
\end{description}

In this case, we say that the family $\{\tau_x\}_{x \in \bb Z^d}$ is ergodic and invariant under $P$. Let $a=(a_1,\dots,a_d): \Omega \to \bb R^d$ be an $\mc F$-measurable function. Assume that there exists $\epsilon_0>0$ such that
\[
\epsilon_0 \leq a_i(\omega) \leq \epsilon_0^{-1} \text{ for all } \omega \in \Omega \text{ and } i=1,\dots,d.
\] 

We say in this situation that the environment satisfies the {\em ellipticity condition}. Fix $\omega \in \Omega$. Define $\omega^n$ by $\omega^n_{x,x+e_i/n}=\omega^n_{x+e_i/n,x}= n^2 a_i(\tau_nx \omega)$, $\omega^n_{x,y}=0$ if $|y-x| \neq 1/n$. Here $\{e_i\}_i$ is the canonical basis of $\bb Z^d$. In this case, $a_n = n^d$ and $\mu$ is the Lebesgue measure in $\bb R^d$.
In \cite{PV}, it is proved that there is a positive definite matrix $A$ such that the quadratic form $\mc E_n$ associated to $\omega^n$ is $\Gamma$-convergent to $\mc E(f)=\int \nabla f \cdot A \nabla f dx$, $P-a.s.$ In particular, Theorem \ref{t1} applies with $\mc L f= \text{div}(A \nabla f)$. This result was first obtained in \cite{GJ}. 

\subsection{The percolation cluster}

Let $e=\{e^i_x; x \in \bb Z^d, i=1,\cdots,d\}$ be a sequence of i.i.d. random variables, with $P(e_x^i=1) = 1- P(e_x^i=0) =p$ for some  $p=(0,1)$. Define for $x,y  \in X_n$, $\omega^n_{x,x+e_i/n}=\omega^n_{x+e_i/n,x}= n^2 e_{nx}^i$, $\omega^n_{x,y}=0$ if $|y-x| \neq 1/n$. Fix a realization of $e$. We say that two points $x, y \in X_n$ are connected if there is a finite sequence $\{x_0=x,\dots,x_l=y\} \subseteq X_n$ such that $|x_{i-1}-x_i|=1/n$ and $\omega_{x_{i-1},i}^n =1$ for any $i$. Denote by $\mc C_0$ the set of points connected to the origin. It is well known that there exists $p_c \in (0,1)$ such that $\theta(p)= P(\mc C_0 \text{ is infinite })$ is 0 for $p<p_c$ and positive for $p>p_c$. Fix $p>p_c$. Define $a_n = n^d$ and $\mu_0(dx) = \theta(p) dx$. In \cite{F2}, it is proved that there exists a constant $D$ such that, $P-a.s$ in the set $\{\mc C_0 \text{ is infinite }\}$, the quadratic form $\mc E_n$ associated to the environment $\omega^n$ restricted to $\mc C_0$ is $\Gamma$-convergent to $\mc E(f) = \theta(p) D \int (\nabla f)^2 dx$. Theorem \ref{t1} applies with $\mc L = D \Delta$, assuming that the initial measures $\nu_n$ put mass zero in configurations with particles outside $\mc C_0$.
This result was first obtained in \cite{Fag}, relying on a duality representation of the simple exclusion process.

\subsection{One-dimensional, inhomogeneous environments} 

In dimension $d=1$, the $\Gamma$-convergence of $\mc E_n$ can be studied explicitely. For nearest-neighbors environments ($\omega^n_{x,y} =0$ if $|x-y|=1$), $\Gamma$-convergence of $\mc E_n$ is equivalent to convergence in distribution of the measures
\[
W_n(dx) = \frac{1}{n} \sum_{x \in \bb Z} (\omega_{x,x+1}^n)^{-1} \delta_{x/n}(dx).
\] 

Let $W(dx)$ be the limit. We assume that $W(dx)$ gives positive mass to any open set. For simplicity, suppose that $W(\{0\})=0$. Otherwise, we simply change the origin to another point with mass zero. For two functions $f,g: \bb R \to \bb R$ we say that $g = df/dW$ if
\[
f(x) = f(0) + \int_0^x g(y) W(dy).
\]
Then $\mc E_n$ is $\Gamma$-convergent to the quadratic form defined by $\mc E(f) = \int (df/dW)^2 dW$. In this case, $\mc L = d/dx d/dW$. A technical difficulty appears if $W(dx)$ has atoms. In that case, there is no kernel $\mc K$ for $\mc L$ contained in $\mc C_c(\bb R)$. To overcome this point, we define for $x \leq y$, $d_W(x,y)=d_W(y,x)= W((x,y])$. The function $d_W$ is a metric in $\bb R$, and in general $\bb R$ is {\em not} complete under this metric: an increasing  sequence $x_n$ converging to $x$ is always a Cauchy sequence with respect to $d_W$, but $d_W(x_n,x)\geq W(\{x\})$, which is non-zero if $x$ is an atom of $W$. Define $\bb R_W = \bb R \cup \{x-; W(\{x\})>0\}$. It is easy to see that $\bb R_W$ is a complete, separable space under the natural extension of $d_W$, and that continuous functions in $\bb R_W$ are in bijection with  c\`adl\`ag functions in $\bb R$ with discontinuity points contained on the set of atoms of $W(dx)$. It is not difficult to see that the set of $W$-differentiable functions in $\mc C_c(\bb R_W)$ is a kernel for $\mc L$ and that Theorems \ref{t1} and \ref{t3} apply to this setting. In \cite{FJL}, the remarkable case on which $W(dx)$ is a {\em random}, self-similar measure (an $\alpha$-stable subordinator) was studied in great detail. 

\subsection{Finitely ramified fractals} 

Let us consider the following sequence of graphs in $\bb R^2$. Define $a_0=(0,0)$, $a_1=(1/2,\sqrt 3/2),$ and  $a_2= (1,0)$ and define $\varphi_i: \bb R^2 \to \bb R^2$ by taking $\varphi_i(x) = (x+a_i)/2$. Define $X_0= \{a_0,a_,a_2\}$ and $X_{n+1}= \cup_i \varphi_i(X_n)$ for $n \geq 0$. For $x,y \in X_0$ we define $\omega^0_{x,y}=1$, we put $\omega_{x,y}^0=0$ if $\{x,y\} \subsetneq X$ and inductively we define
\[
\omega_{x,y}^{n+1} = 5 \sum_i \omega^n_{\varphi^{-1}_i(x), \varphi^{-1}(y)}.
\]

The set $X_n$ is  a discrete approximation of the Sierpinski gasket $X$ defined as the unique compact, non-empty set $X$ such that $X = \cup_i \varphi_i(X)$. Here we are just saying that $\omega_{x,y}^n = 5^n$ if $x,y$ are neighbors in the canonical sense. In this case $a_n=3^n$ and $\mu$ is the Hausdorff measure in $X$. It has been proved \cite{Kig} that the quadratic forms $\mc E_n$ converge to a certain Dirichlet form $\mc E$ which is used to define an abstract Laplacian in $X$. In particular, Theorems \ref{t1} and \ref{t3} apply to this model. This result was obtained in \cite{Jar3} in the context of a zero-range process. The same result can be proved for general {\em finitely ramified fractals}, in the framework of \cite{Kig}.


\begin{thebibliography}{18}
\providecommand{\natexlab}[1]{#1}
\providecommand{\url}[1]{\texttt{#1}}
\providecommand{\urlprefix}{URL }
\expandafter\ifx\csname urlstyle\endcsname\relax
  \providecommand{\doi}[1]{doi:\discretionary{}{}{}#1}\else
  \providecommand{\doi}{doi:\discretionary{}{}{}\begingroup
  \urlstyle{rm}\Url}\fi
\providecommand{\eprint}[2][]{\url{#2}}

\bibitem[{Faggionato(2007)}]{Fag}
A.~Faggionato.
\newblock Bulk diffusion of 1{D} exclusion process with bond disorder.
\newblock \emph{Markov Process. Related Fields} \textbf{13}~(3), 519--542
  (2007).
\newblock ISSN 1024-2953.

\bibitem[{Faggionato et~al.(2009)Faggionato, Jara and Landim}]{FJL}
A.~Faggionato, M.~Jara and C.~Landim.
\newblock Hydrodynamic behavior of 1{D} subdiffusive exclusion processes with
  random conductances.
\newblock \emph{Probab. Theory Related Fields} \textbf{144}~(3-4), 633--667
  (2009).
\newblock ISSN 0178-8051.

\bibitem[{Faggionato(2008)}]{Fag2}
Alessandra Faggionato.
\newblock Random walks and exclusion processes among random conductances on
  random infinite clusters: homogenization and hydrodynamic limit.
\newblock \emph{Electron. J. Probab.} \textbf{13}, 2217--2247 (2008).
\newblock ISSN 1083-6489.

\bibitem[{Faggionato and Martinelli(2003)}]{FM}
Alessandra Faggionato and Fabio Martinelli.
\newblock Hydrodynamic limit of a disordered lattice gas.
\newblock \emph{Probab. Theory Related Fields} \textbf{127}~(4), 535--608
  (2003).
\newblock ISSN 0178-8051.

\bibitem[{Franco and Landim(2008)}]{FL}
F.~Franco and C.~Landim.
\newblock Hydrodynamic limit of gradient exclusion processes with conductances.
\newblock \emph{To appear in Arch. Rat. Mech. Anal.}  (2008).
\newblock \urlprefix\url{http://arxiv.org/abs/0806.3211}.

\bibitem[{Fritz(1989)}]{Fri}
J.~Fritz.
\newblock Hydrodynamics in a symmetric random medium.
\newblock \emph{Comm. Math. Phys.} \textbf{125}~(1), 13--25 (1989).
\newblock ISSN 0010-3616.

\bibitem[{Gon{\c{c}}alves and Jara(2008{\natexlab{a}})}]{GJ}
Patr{\'{\i}}cia Gon{\c{c}}alves and Milton Jara.
\newblock Scaling limits for gradient systems in random environment.
\newblock \emph{J. Stat. Phys.} \textbf{131}~(4), 691--716
  (2008{\natexlab{a}}).
\newblock ISSN 0022-4715.

\bibitem[{Gon{\c{c}}alves and Jara(2008{\natexlab{b}})}]{GJ2}
Patr{\'{\i}}cia Gon{\c{c}}alves and Milton Jara.
\newblock Scaling limits of a tagged particle in the exclusion process with
  variable diffusion coefficient.
\newblock \emph{J. Stat. Phys.} \textbf{132}~(6), 1135--1143
  (2008{\natexlab{b}}).

\bibitem[{Guo et~al.(1988)Guo, Papanicolaou and Varadhan}]{GPV}
M.~Z. Guo, G.~C. Papanicolaou and S.~R.~S. Varadhan.
\newblock Nonlinear diffusion limit for a system with nearest neighbor
  interactions.
\newblock \emph{Comm. Math. Phys.} \textbf{118}~(1), 31--59 (1988).
\newblock ISSN 0010-3616.

\bibitem[{Jara and Landim(2006)}]{JL}
M.~D. Jara and C.~Landim.
\newblock Nonequilibrium central limit theorem for a tagged particle in
  symmetric simple exclusion.
\newblock \emph{Ann. Inst. H. Poincar\'e Probab. Statist.} \textbf{42}~(5),
  567--577 (2006).

\bibitem[{Jara(2009)}]{Jar3}
Milton Jara.
\newblock Hydrodynamic limit for a zero-range process in the {S}ierpinski
  gasket.
\newblock \emph{Comm. Math. Phys.} \textbf{288}~(2), 773--797 (2009).
\newblock ISSN 0010-3616.

\bibitem[{Kigami(2001)}]{Kig}
Jun Kigami.
\newblock \emph{Analysis on fractals}, volume 143 of \emph{Cambridge Tracts in
  Mathematics}.
\newblock Cambridge University Press, Cambridge (2001).
\newblock ISBN 0-521-79321-1.

\bibitem[{Kipnis and Landim(1999)}]{KL}
Claude Kipnis and Claudio Landim.
\newblock \emph{Scaling limits of interacting particle systems}, volume 320 of
  \emph{Grundlehren der Mathematischen Wissenschaften [Fundamental Principles
  of Mathematical Sciences]}.
\newblock Springer-Verlag, Berlin (1999).
\newblock ISBN 3-540-64913-1.

\bibitem[{Koukkous(1999)}]{Kou}
A.~Koukkous.
\newblock Hydrodynamic behavior of symmetric zero-range processes with random
  rates.
\newblock \emph{Stochastic Process. Appl.} \textbf{84}~(2), 297--312 (1999).
\newblock ISSN 0304-4149.

\bibitem[{Nagy(2002)}]{Nag}
Katalin Nagy.
\newblock Symmetric random walk in random environment in one dimension.
\newblock \emph{Period. Math. Hungar.} \textbf{45}~(1-2), 101--120 (2002).
\newblock ISSN 0031-5303.

\bibitem[{Papanicolaou and Varadhan(1982)}]{PV}
George~C. Papanicolaou and S.~R.~S. Varadhan.
\newblock Diffusions with random coefficients.
\newblock In \emph{Statistics and probability: essays in honor of {C}. {R}.
  {R}ao}, pages 547--552. North-Holland, Amsterdam (1982).

\bibitem[{Quastel(2006)}]{Qua}
Jeremy Quastel.
\newblock Bulk diffusion in a system with site disorder.
\newblock \emph{Ann. Probab.} \textbf{34}~(5), 1990--2036 (2006).
\newblock ISSN 0091-1798.

\bibitem[{Tartar(1979)}]{Tar}
L.~Tartar.
\newblock Homog\'en\'eisation et compacit\'e par compensation.
\newblock In \emph{S\'eminaire {G}oulaouic-{S}chwartz (1978/1979)}, pages Exp.
  No. 9, 9. \'Ecole Polytech., Palaiseau (1979).

\end{thebibliography}
\end{document}